\documentclass[a4paper,12pt,leqno]{amsart}
\usepackage{geometry}
\usepackage[utf8]{inputenc}
\usepackage{amssymb,amsfonts,amsmath,amsthm,eucal,url,cite}
\usepackage{enumitem}
\usepackage{tikz-cd}
\usepackage{hyperref}
\usepackage{multicol}
\usepackage{slashed}

\newcommand{\dd}{\mathrm{d}}

\newcommand{\RR}{\mathbb{R}}

\def\Span{\mathrm{Span}}

\newtheorem{theorem}{Theorem}[section]

\newtheorem{lemma}[theorem]{Lemma}
\newtheorem{definition}[theorem]{Definition}

\theoremstyle{definition}

\theoremstyle{remark}
\newtheorem{rem}[theorem]{\bf Remark}

\title{Conformal vector fields on LCP manifolds}
\author{Brice Flamencourt, Andrei Moroianu}

\address{Brice Flamencourt, UMPA - ENS Lyon, CNRS, 46 allée d’Italie, 69364 Lyon, France.}
\email{brice.flamencourt@ens-lyon.fr}

\address{Andrei Moroianu \\ Université Paris-Saclay, CNRS,  Laboratoire de mathématiques d'Orsay, 91405, Orsay, France, 
and Institute of Mathematics “Simion Stoilow” of the Romanian Academy, 21 Calea Grivitei, 010702 Bucharest, Romania}
\email{andrei.moroianu@math.cnrs.fr}


\begin{document}
\begin{abstract}
    We show that conformal vector fields on compact locally conformally product manifolds are orthogonal to the flat distribution and Killing with respect to the Gauduchon metric.
\end{abstract}

\maketitle

\begin{center}
\end{center}
\section{Introduction}

Locally conformally product (LCP) structures arise naturally in the theory of Weyl connections. They consist in a closed non-exact Weyl connection $D$ on a compact conformal manifold $(M,c)$ with (non-zero) reducible holonomy. LCP structures are closely related to LCK structures, which can be defined in a similar way, just by replacing the reducibility condition for the holonomy of $D$ with the existence of a $D$-parallel complex structure.

The study of LCP structures started with the Belgun-Moroianu conjecture \cite{BM}, which stated that they should not exist. However, the works of Matveev-Nikolayevsky \cite{MN15,MN17} and Kourganoff \cite{Kou} after them exhibited examples of such manifolds, and also showed that their structure was very special. Indeed, when lifted to the universal cover $\tilde M$ of $M$, $D$ becomes the Levi-Civita connection of a Riemannian metric $h$ and $(\tilde M, h)$ is isometric to $(\RR^q,h_0) \times (N, h_N)$, where $(N,h_N)$ is an irreducible Riemannian manifold. 

After the discovery of this result, LCP manifolds have been analyzed from various points of view by several authors. LCP structures on solvmanifolds \cite{ABM} and more generally on compact quotients of Lie groups \cite{dBM} have been studied by Andrada, del Barco and the second author, and their description under some additional assumptions concerning their characteristic group has been given by the first author in \cite{F24}. Obstructions to the existence of LCP structures on conformal manifolds are also discussed in \cite{BFM}, where it is shown that the conformal class of an LCP manifold cannot contain Einstein of Kähler metrics for example.

An important tool in conformal geometry is the Gauduchon metric \cite{G}, which on a compact conformal manifold $(M,c)$ endowed with a Weyl connection $D$ is the unique (up to a multiplicative constant)  Riemannian metric $g\in c$ with the property that the Lee form of $D$ with respect to $g$ is coclosed. 

Recently, the second author together with Pilca Pilca \cite{MP24} studied the behavior of the Gauduchon metric of Weyl connections defining an LCP structure, and proved that it is adapted. Recall that if $(M,c,D)$ is an LCP manifold, a Riemannian metric $g$ in the conformal class $c$ is called adapted if the lift to the universal cover of the Lee form of $D$ with respect to $g$ vanishes on the flat distribution $T \RR^q$. 

In another recent work \cite{MP23}, the same authors generalized to the LCK setting the well-known fact that all conformal vector fields on a compact Kähler manifold are Killing, by proving that every conformal vector field on a compact LCK manifold is Killing with respect to the Gauduchon metric. 

It is thus natural to investigate the corresponding question in the LCP setting, all the more that a classical result of Tashiro and Miyashita \cite{TM67} states that on complete non-flat Riemannian products, all complete conformal vector fields are Killing. 

This is precisely what we achieve in this paper, by proving the following:

\begin{theorem} \label{main}
Let $(M,c,D)$ be an LCP manifold. Then the conformal vector fields of $(M,c)$ are exactly the Killing vector fields of the Gauduchon metric of the Weyl structure $D$.
\end{theorem}

The strategy of the proof is roughly as follows. Every conformal vector field $\bar\xi$ on $(M,c)$ can be lifted to a (complete) conformal vector field $\xi$ of the Riemannian product structure $(\RR^q,h_0) \times (N, h_N)$ on the universal cover of $M$. However, the Tashiro-Miyashita result does not apply, since this product metric is not complete. Nonetheless, we can show that the covariant derivative of $\xi$ with respect to tangent vectors to $N$, restricted to each flat leaf $\RR^q\times \{y\}$, is a gradient conformal vector field. Then using the explicit description of these vector fields and the fact the $\xi$ has bounded geometry, we obtain that $\xi$ has to be the pull-back of a Killing vector field on $(N,h_N)$. This implies in particular that $\bar \xi$ is affine with respect to the Weyl connection $D$, and we conclude by the uniqueness property of the Gauduchon metric.

{\bf Acknowledgments.} This work was partly supported by the PNRR-III-C9-2023-I8 grant CF 149/31.07.2023 {\em Conformal Aspects of Geometry and Dynamics} and by the Procope Project No. 57650868 (Germany) / 48959TL (France).

\section{Preliminaries on LCP structures} \label{preliminaries}

We start by introducing the basic concepts of conformal geometry.

\begin{definition}
A conformal structure on a smooth manifold $M$ is an equivalence class of the equivalence relation defined on the space of Riemannian metrics on $M$ by $g\sim g'$ if and only if there exists a smooth function $\varphi$ such that $g'=e^{2\varphi}g$.
 \end{definition}

Conformal structures are usually denoted by $c$. There is no natural connection on a conformal manifold as for the Riemannian case. However, one can consider a class of connections, called Weyl connections, which extends in a certain way the concept of Levi-Civita connection to the conformal case.

\begin{definition}
Let $(M, c)$ be a conformal manifold. A Weyl connection on $(M,c)$ is a torsion-free connection $D$ on $TM$ which preserves the conformal structure, in the sense that for any metric $g \in c$, there is a $1$-form $\theta_g$, called the {\em Lee form} of $D$ with respect to $g$, such that $D g = - 2 \theta_g \otimes g$.
 \end{definition}

Clearly if $g'=e^{2\varphi}g$ then $\theta_{g'}=\theta_g-d\varphi$, so the cohomological nature of the Lee form is independent of the metric. 
This motivates the following definition:

\begin{definition}
A Weyl connection $D$ on a conformal manifold $(M,c)$ is called {\em closed} if the Lee form of $D$ with respect to one metric - and then to all metrics - in $c$ is closed. Similarly, $D$ is called {\em exact} if the Lee form of $D$ with respect to one metric - and then to all metrics - in $c$ is exact.
\end{definition}

Note that every closed Weyl connection on a simply connected conformal manifold is exact.

Let now $D$ be a closed Weyl connection on a conformal manifold $(M,c)$. Its pull-back $\tilde D$ to the universal covering $\tilde M$ of $M$ is a Weyl connection for the conformal structure $\tilde c$ obtained by pulling-back $c$. This Weyl connection is exact since $\tilde M$ is simply connected, thus there exists a metric $h \in \tilde c$, unique up to a multiplication by a constant, such that $\nabla^h = \tilde D$, where $\nabla^h$ is the Levi-Civita connection of $h$. 

More precisely, if $g$ is any Riemannian metric on $M$ in the conformal class $c$, and $\theta_g$ is the Lee form of $D$ with respect to $g$, then its pull-back $\tilde\theta_g$ to the universal cover $\tilde M$ is exact, so there exists a function $\varphi\in C^\infty(\tilde M)$ such that $d\varphi=\tilde\theta_g$. Moreover, $\tilde\theta_g$ is also the Lee form of the pull-back Weyl connection $\tilde D$ on $\tilde M$, so by the above formula, the Lee form of $\tilde D$ with respect to the metric $h:=e^{2\varphi}\tilde g$ vanishes, i.e. $\tilde D$ is the Levi-Civita connection of $h$.
Note that the fundamental group $\pi_1(M)$ acts on $\tilde M$ by $h$-homotheties, and this action preserves $h$ if and only if $D$ is exact. 

LCP structures arise when one considers closed, non-exact Weyl connections on a compact conformal manifold. In this situation, one has a remarkable result proved by Kourganoff \cite{Kou}:

\begin{theorem} \label{fundLCP} {\rm \cite[Thm. 1.5]{Kou}}
Let $(M,c)$ be a conformal manifold endowed with a closed, non-exact Weyl connection $D$. Let $h$ be a metric on $\tilde M$, the universal cover of $M$, such that $\nabla^h = \tilde D$ where $\tilde D$ is the pull-back of $D$ to $\tilde M$. Then, one of the three following cases occurs:
\begin{itemize}
\item $(\tilde M, h)$ is flat;
\item $(\tilde M, h)$ is irreducible;
\item $(\tilde M, h)$ is a Riemannian product $(\RR^q,h_0) \times (N, h_N)$ where $q \ge 1$, $(\RR^q, h_0)$ is the usual Euclidean space and $(N, h_N)$ is a non-flat, incomplete Riemannian manifold.
\end{itemize}
\end{theorem}

The third case in Theorem~\ref{fundLCP} corresponds to LCP structures. More precisely, we have the following:

\begin{definition}
An LCP manifold is a triple $(M,c,D)$ where $M$ is a compact manifold, $c$ is a conformal structure on $M$ and $D$ is a closed, non-exact Weyl connection, which is non-flat and reducible (i.e. the representation of its restricted holonomy group $\mathrm{Hol}_0(D)$ is reducible).
\end{definition}

With the notations above, $(\RR^q,h_0)$ is called the {\em flat part} of the LCP manifold, while $(N,h_N)$ is called the non-flat part. The distributions $T \RR^q$ and $T N$ descend to $D$-parallel distributions on $M$, respectively called the {\em flat distribution} and the {\em non-flat distribution} of the LCP manifold.

\begin{definition}
    Let $(M,c,D)$ be an LCP manifold. A Riemannian metric $g\in c$ is called {\em adapted} if the Lee form $\theta_g$ vanishes on the flat distribution (or equivalently, if the primitive $\varphi$ of the pull-back to the universal cover of $\theta_g$ is the pull-back of a function defined on the non-flat factor $N$.
\end{definition}

\begin{rem}\label{ad}
    Adapted metrics always exist (see \cite{FlaLCP,MP24}). The above considerations show that their lift to $\tilde M$ can be written $\tilde g=e^{-2\alpha}h$, where $\alpha\in C^\infty(N)$.
\end{rem}

\section{Conformal vector fields on LCP manifolds}
Let $(M, c, D)$ be an $n$-dimensional LCP manifold. Let $h$ be the Riemannian metric on its universal cover whose Levi-Civita connection is the lift of $D$, so that  $(\tilde M, h) = (\RR^q, h_0) \times (N, h_N)$, where $(\RR^q, h_0)$ is the flat part and $(N, h_N)$ is the non-flat part. We will denote the Levi-Civita connection of $h$ by $\nabla$.

Denoting by $p_1$, $p_2$ the projections from $\tilde M$ onto $\RR^q$ and $N$ respectively, the tangent bundle of $\RR^q\times N$ can be identified with the direct sum $\pi_1^* T\RR^q\oplus \pi_2^*TN$. Correspondingly, one can write every vector field $\xi\in\Gamma(T\tilde M)$ as a sum 
\begin{equation}\label{xi}
    \xi=\xi_1+\xi_2,
\end{equation} 
where $\xi_1(\cdot,y)\in\Gamma(T\RR^q)$ for every $y\in N$, and $\xi_2(x,\cdot)\in\Gamma(TN)$ for every $x\in \RR^q$.

Let $\bar\xi \in \Gamma(TM)$ be a conformal vector field, and let $\xi \in \Gamma(T \tilde M)$ be its lift to $\tilde M$. In particular $\gamma^* \xi = \xi$ for any $\gamma \in \pi_1(M)$. We decompose $\xi = \xi_1 + \xi_2$ as in \eqref{xi}. The fact that $\xi$ is a conformal vector field on $\tilde M$ is equivalent to the identity
\begin{align} \label{confeq}
h (\nabla_X \xi, Y) + h (X, \nabla_Y \xi) = f h(X, Y), && (\forall) X, Y \in T \tilde M,
\end{align}
where $f$ is a real-valued smooth function on $\tilde M$. Applying \eqref{confeq} with $X, Y \in T \RR^q$ shows that for any $y \in N$, $\xi_1 (\cdot, y)$ is a conformal vector field on $\RR^q$. In the same way, we show that for any $x \in \RR^q$, $\xi_2 (x, \cdot)$ is a conformal vector field on $N$.

Let $(e_1, \ldots, e_q)$ be the canonical basis of $\RR^q$ and denote by the same letters the induced vector fields on $\tilde M$ constant along $N$. We also fix an arbitrary vector field $Z$ on $N$, identified with the induced vector field on $\tilde M$ constant along $\RR^q$. Then $\nabla_{e_i}Z=\nabla_Ze_i=0$ for every $1\le i \le q$. 

Using \eqref{confeq} for $X=e_i$ and $Y=Z$ we obtain
\begin{equation}\label{nzx}
 h (e_i, \nabla_Z \xi_1) = -h (\nabla_{e_i} \xi_2, Z)=-\partial_{e_i}(h(\xi_2,Z)),
\end{equation}
which implies that for every fixed $y\in N$, the vector field $\nabla_Z \xi_1(\cdot,y)$ on $\RR^q$ is the gradient in $\RR^q$ of the function $-h_{(\cdot,y)}(\xi_2,Z)$.

Moreover, taking $X=e_i$, $Y=e_j$ in \eqref{confeq}, differentiating with respect to $Z$, and using the commutation of $\nabla_Z$ with $\nabla_{e_i}$ for $1\le i\le q$, shows that $\nabla_Z \xi_1 (\cdot, y)$ is a conformal vector field on $(\RR^q,h_0)$ for every $y\in N$.

Synthesizing the previous analysis, for any $y \in N$ and $Z \in T_y N$, $\nabla_Z \xi_1 (\cdot, y)$ is a gradient conformal vector field on $(\RR^q,h_0)$. These vector fields are well understood:

\begin{lemma} \label{EucConfGradField}
Let $X$ be a gradient conformal vector field on $(\RR^q,h_0)$. Then, there exist real numbers $b, b_1, \ldots, b_q$ such that $X = b \sum_{i=1}^qx_i e_i + \sum_{i=1}^qb_i e_i$, where $(e_1, \ldots, e_q)$ is the canonical basis of $\RR$ and $x_i$ is the $i$-th coordinate function.
\end{lemma}
\begin{proof}
In this proof, $\nabla$ stands for the gradient in $\RR^q$ and the usual scalar product on $\RR^q$ is denoted by $\langle \cdot, \cdot \rangle$.

By assumption, there is a function $\psi : \RR^q \to \RR$ such that $X = \nabla \psi=\sum_{i=1}^q(\partial_{e_i} \psi) e_i$. Since $X$ is a conformal vector field one has
\begin{align*}
\langle \partial_{e_i} X, e_j \rangle + \langle e_i, \partial_{e_j} X \rangle = \chi \langle e_i, e_j \rangle, && (\forall) 1 \le i, j \le q,
\end{align*}
for some smooth function $\chi : \RR^q \to \RR$. This is equivalent to
\begin{align*}
\partial_{e_i} \partial_{e_j} \psi = \delta_i^j \chi, && (\forall) 1 \le i, j \le q.
\end{align*}
In particular, for any $1 \le i \le q$, $\partial_{e_i} \psi$ is a function depending only on the $i$-th coordinate, and thus $\chi$ has the same property. Consequently, $\chi$ is constant and $\psi$ belongs to the vector space
\begin{equation}
E := \lbrace \phi \in C^\infty(\RR^q) \ \vert \ \exists \mu \in \RR, \ \forall i,j, \ \partial_{e_i} \partial_{e_j} \phi  = \delta_i^j \mu \rbrace.
\end{equation}
Now, we remark that
\begin{equation}
E_0 := \lbrace \phi \in C^\infty(\RR^q) \ \vert \ \ \forall i,j, \ \partial_{e_i} \partial_{e_j} \phi  = 0 \rbrace
\end{equation}
is a hyperplane of $E$ since it is the kernel of the linear form $E \ni \phi \mapsto \partial_{e_1}^2 \phi$. Moreover, denoting by $x_i$ the $i$-th coordinate function, $\Span (\sum_{i=1}^q(x_i)^2)$ is a supplementary of $E_0$ in $E$, whence $$E  = \Span (\sum_{i=1}^q(x_i)^2) \oplus E_0.$$

Since $E_0 = \Span (x_1, \ldots, x_q)$, we obtain $\psi = \frac{b}{2} \sum_{i=1}^q(x_i)^2 +\sum_{i=1}^qb_i x_i$ for some $(b, b_1, \ldots, b_q) \in \RR^{q+1}$, and $X = \nabla \psi = b \sum_{i=1}^qx_i e_i + \sum_{i=1}^qb_i e_i$.

Conversely, any vector field of this form is a conformal gradient vector field.
\end{proof}

From Lemma~\ref{EucConfGradField} we conclude that there are functions $\tilde b, \tilde b_1, \ldots, \tilde b_q$ from $T N$ to $\RR$ such that
\begin{equation}  \label{nz}
\nabla_Z \xi_1 (x, y) = \tilde b (y, Z)\sum_{i=1}^q x_i e_i + \sum_{i=1}^q\tilde b_i (y, Z) e_i,
\end{equation}
for all $(x,y) \in \RR^q \times N$, and $Z \in T_y N$.
Applying $\eqref{nz}$ to $x=0$ and $x=e_1$ shows that
\begin{align}
\tilde b_i (y, Z) = h_{ (0, y)} (\nabla_Z \xi_1, e_i), && \tilde b (y, Z) = h_{(e_1, y)} (\nabla_Z \xi_1, e_1) - \tilde b_1 (y, Z)
\end{align}
for any $y \in  N$, $Z \in T_y N$ and $1 \le i \le q$, which implies that the functions $\tilde b, \tilde b_1, \ldots, \tilde b_q$ are smooth, and linear in the variable $Z$.

We fix $y_0 \in N$. Let $y \in N$ and $c : [0, 1] \to N$ be a  smooth path from $y_0$ to $y$, which exists by connectedness. One has, for any $x \in \RR^q$,
\begin{align*}
\xi_1 (x, y) - \xi_1 (x, y_0) &= \int_0^1 \nabla_{\dot c (t)} \xi_1 (x, c(t)) \dd t \\
&= \int_0^1 \tilde b(c(t), \dot c (t)) \sum_{i=1}^qx_i e_i + \sum_{i=1}^q\tilde b_i(c(t), \dot c (t)) e_i \dd t \\
&= \left( \int_0^1 \tilde b (c(t), \dot c (t)) \dd t \right) \sum_{i=1}^qx_i e_i + \sum_{i=1}^q\left( \int_0^1 \tilde b_i (c(t), \dot c (t)) \dd t \right) e_i.
\end{align*}
We define the smooth functions
\begin{align}
b (y) := \int_0^1 \tilde b (c(t), \dot c (t)) \dd t, && b_i (y) := \int_0^1 \tilde b_i (c(t), \dot c (t)) \dd t,
\end{align}
which by the above computation do not depend on the chosen path. Then, writing $\xi_1 (\cdot, y_0) =: \sum_{i=1}^q\beta_i e_i$ with $\beta_i \in C^\infty (\RR^q)$, we have
\begin{align}\label{xi1}
\xi_1 (x, y) = \sum_{i=1}^q(\beta_i (x) e_i + b (y) x_i e_i + b_i (y) e_i), && (\forall) (x, y) \in \RR^q \times N.
\end{align}
Using \eqref{nzx}, one obtains for every $Z\in \Gamma(TN)$ and $1\le i\le q$:
\[
\partial_{e_i}(h(\xi_2,Z))=-(Z(b)x_i+Z(b_i)),
\]
which implies that there exists some vector field $V \in \Gamma (T N)$ such that 
\begin{equation} \label{formXi2}
\xi_2 (x, y) = V(y) - \sum\limits_{i = 1}^q \left(\frac{(x_i)^2}{2}\,\nabla^N b(y)  + x_i\,\nabla^N b_i(y) \right)
\end{equation}
for all $(x,y)\in\RR^q\times N$, where $\nabla^N$ denotes the gradient defined by the Levi-Civita connection of $h_N$.

Let $g$ be an adapted metric on $N$ (Remark \ref{ad}) whose pull-back to $\tilde M$ can be written as $\tilde g = e^{-2 \alpha} h$ for some $\alpha \in C^\infty (N)$. Note that the function $\alpha$ is not bounded from above or from below. Indeed, there exists an element $\gamma\in\pi_1(M)$ which is a contraction with respect to $h$ (i.e. $\rho^*h=\lambda h $ with $\lambda\in (0,1)$), and since $\pi_1(M)$ acts isometrically with respect to $\tilde g$, we get $\gamma^*\alpha=\alpha+\frac12\ln\lambda$.

Since $M$ is compact, there exists a constant $C > 0$ such that
\[
\underset{\tilde M}{\sup} \Vert \xi \Vert_{\tilde g} = \underset{M}{\sup} \Vert \bar \xi \Vert_g \le C.
\]
Consequently,
\[
\Vert e^{-\alpha} \xi_2 \Vert^2_h = \Vert \xi_2 \Vert^2_{\tilde g} \le \Vert \xi \Vert^2_{\tilde g} \le C^2,
\]
whence for any $y \in N$, the estimate 
\begin{equation}\label{est}
    \Vert \xi_2 (x, y) \Vert_h \le e^{ \alpha (y)} C
\end{equation} 
(independent on $x$) holds. 
Now, Equation~\eqref{formXi2} shows that for every fixed $y_0\in N$, the square norm $\Vert \xi_2 (x,y_0)\Vert_h $ is polynomial in $(x_1, \ldots, x_q)$. Clearly this norm is bounded on $\RR^q$ if and only if $\nabla^N b(y_0) = \nabla^N b_i(y_0) = 0$ for every $1\le i\le q$ and for every $y_0\in N$, showing that the functions $b, b_1, \ldots, b_q$ are constant on $N$. Consequently, $\xi_2(x,y)=V(y)$ is induced on $\tilde M$ by a conformal vector field of $(N,g_N)$.

On the other hand, one also has $\Vert e^{-\alpha} \xi_1 \Vert^2_h = \Vert \xi_1 \Vert^2_{\tilde g} = \Vert \xi_1 \Vert^2_{\tilde g} \le \Vert \xi \Vert^2_{\tilde g} \le C^2$. Since $b, b_1, \ldots, b_q$ are constant, \eqref{xi1} shows that $\xi_1$ depends only on the variable $x$ of $\RR^q$. Therefore, for every $x\in\RR^q$ and $y\in N$ one has $\Vert \xi_1 (x) \Vert_{h} \le e^{- \alpha(y)} C$.
However, we have seen that the function $\alpha$ is unbounded from above. We conclude that $\Vert \xi_1 (x) \Vert_{h} = 0$, so finally $\xi_1 = 0$.

The conformal vector field $\xi=\xi_2=V$ is thus tangent to $N$ and constant in the direction of $\RR^q$. Taking non-zero vectors $X = Y \in T \RR^q$ in Equation~\eqref{nzx}, one gets
\[
f h(X,X) = 0,
\]
which in turn implies $f = 0$, so $\xi$ is a Killing vector field for $h$.

Altogether, we have proved the following result:

\begin{theorem} \label{ConfFieldsStruct}
Let $(M,c,D)$ be an LCP manifold. Then the lift of any conformal vector field of $(M, c)$ to the universal cover of $M$ endowed with its canonical Riemannian decomposition $(\RR^q,h_0) \times (N, h_N)$ is a Killing vector field of $(N, h_N)$.
\end{theorem}

The last step in our analysis of conformal vector fields on LCP manifolds is to link them to the Killing fields of the Gauduchon metric \cite{G} of the Weyl structure $(M,c,D)$. We first recall the definition of this particular metric:

\begin{definition}
If $(M,c)$ is a compact conformal manifold of dimension larger than $2$, then for any Weyl connection $D$ on $(M,c)$ there exists a unique (up to a positive multiplicative constant) Riemannian metric $g_G$ such that the Lee form of $D$ with respect to $g_G$ is coclosed. This metric is called the {\em Gauduchon metric} of the Weyl structure.
\end{definition}

We now return to our particular setting, with $\bar \xi \in TM$ a conformal vector field of $(M,c)$. By Theorem~\ref{ConfFieldsStruct}, the lift $\xi$ of $\bar \xi$ to $\Tilde M$ is a Killing vector field on $N$. In particular, it is a Killing vector field for $(\Tilde M,h)$, so it preserves the Levi-Civita connection $\nabla$ of $h$. But $\nabla$ is the lift of the Weyl connection $D$, and we deduce that $\bar \xi$ is an affine vector field for $D$, i.e. $\mathcal L_{\bar \xi} (D) = 0$.

Let $g_G$ be the Gauduchon metric of the Weyl structure $(M,c,D)$. We denote by $\theta_G$ the Lee form of $D$ with respect to $g_G$, by $\nabla^{g_G}$ the Levi-Civita connection of $g_G$, and by $(\varphi_t)_{t\in\RR}$ the flow of $\bar \xi$. Since $\bar \xi$ is affine, we obtain for any $t \in \RR$:
\begin{equation}
D (\varphi_t^* g_G) = \varphi_t^* (D g_G) = -2 \varphi_t^*(\theta_G \otimes g_G) = -2 (\varphi_t^* \theta_G) \otimes (\varphi_t^* g_G),
\end{equation}
and we get that $\varphi_t^* \theta_G$ is the Lee form of $D$ with respect to $\varphi_t^* g_G$. The Levi-Civita connection of $g_G$ is $\varphi_t^* \nabla^{g_G}$, hence $\delta^{\varphi_t^* g_G} = \varphi_t^* \delta^{g_G}$. We thus have:
\begin{equation}
\delta^{\varphi_t^* g_G} \theta_{\varphi_t^* g_G} = (\varphi_t^* \delta^{g_G}) (\varphi_t^* \theta_{g_G}) = \varphi_t^* (\delta^{g_G} \theta_{g_G}) = 0.
\end{equation}
Consequently, $\varphi_t^* g_G$ still is a Gauduchon metric of $(M,c,D)$, so  by the uniqueness property there exists a constant $\lambda_t > 0$ such that $(\varphi_t)^* g_G = \lambda_t g_G$. It follows that there exists $\lambda > 0$ such that
\begin{equation}\label{lg}
\mathcal L_{\bar \xi} g_G = \lambda g_G.
\end{equation}
Taking the trace in \eqref{lg} yields $2\delta \bar\xi = - n \lambda$. Integrating this equality over $M$ and using the divergence theorem, one obtains $\lambda = 0$, i.e.
\begin{equation}
\mathcal L_{\bar \xi} g_G = 0,
\end{equation}
thus showing that $\bar \xi$ is a Killing vector field for $g_G$.

Conversely, it is obvious that a Killing vector field for $g_G$ is a conformal vector field of $(M,c)$. This concludes the proof of Theorem~\ref{main}.

\renewcommand{\refname}
\end{document}